\documentclass[12pt,a4paper]{amsart}
\usepackage{hyperref}
\usepackage{graphics}
\usepackage{amsmath}
\textwidth 6.375 true in\newtheorem{theorem}{Theorem}[section]
\newtheorem{remark}[theorem]{Remark}
\newtheorem{corollary}[theorem]{Corollary}
\theoremstyle{theorem}

\newtheorem{lemma}[theorem]{Lemma}
\theoremstyle{definition} 

\numberwithin{equation}{section} \makeatletter
\@namedef{subjclassname@2010}{ 2010 Mathematics Subject
Classification} \makeatother
\title[Some Inequalities for polar derivatives]{Inequalities for the polar derivative of a polynomial}
\author{N. A. Rather and Suhail Gulzar}
 \address{Department of Mathematics \\
    University of Kashmir \\
   Srinagar, Hazratbal 190006
   \\ India}
 \email{drnarather@gmail.com}
 \email{sgmattoo@gmail.com}
\date{}
\begin{document}
\maketitle{}
\begin{abstract}
 Let $ P(z) $ be a polynomial of degree $ n $ and for any real or complex number $\alpha,$ let $D_\alpha P(z)=nP(z)+(\alpha-z)P^{\prime}(z)$ denote the polar derivative with respect to $\alpha.$ In this paper, we obtain generalizations of some inequalities for the polar derivative of a polynomial.
\end{abstract}
\begin{center}
\section{\bf Introduction and statement of results}
\end{center}
\hspace{5mm} Let $P(z) $ be a polynomial of degree $n,$ then
\begin{equation}\label{e1}
\underset{\left|z\right|=1}{Max}\left|P^{\prime}(z)\right|\leq  n\underset{\left|z\right|=1}{Max}\left|P(z)\right|.
\end{equation} 
 Inequality \eqref{e1} is an immediate consequence of S. Bernstein's Theorem on the derivative of a trigonometric polynomial (for reference, see \cite[p.531]{mm}, \cite[p.508]{rs} or \cite{asc}) and the result is best possible with equality holding for the polynomial $P(z)=az^n,$ $a\neq 0.$ \\
\indent  If we restrict ourselves to the class of polynomials having no zero in $|z|<1$, then inequality \eqref{e1} can be replaced by
\begin{equation}\label{e2}
\underset{\left|z\right|=1}{Max}\left|P^{\prime}(z)\right|\leq        \frac{n}{2}\underset{\left|z\right|=1}{Max}\left|P(z)\right|.
\end{equation}
Inequality \eqref{e2} was conjectured by Erd\"{o}s and later verified by Lax \cite{el}. The result is sharp and equality holds for $P(z)=az^n+b,$ $|a|=|b|.$\\
\indent  For the class polynomials  having all zeros in $|z|\leq 1,$ it was proved by Tur\'{a}n \cite{t} that
\begin{equation}\label{e3}
\underset{\left|z\right|=1}{Max}\left|P^{\prime}(z)\right|\geq        \frac{n}{2}\underset{\left|z\right|=1}{Max}\left|P(z)\right|.
\end{equation}
The inequality \eqref{e3} is best possible and become equality for polynomial $P(z)=(z+1)^n.$
As an extension of \eqref{e2} and \eqref{e3} Malik \cite{m} proved that if $P(z)\neq 0$ in $|z|<k$ where $k\geq 1,$ then
\begin{equation}\label{e4}
\underset{\left|z\right|=1}{Max}\left|P^{\prime}(z)\right|\leq        \frac{n}{1+k}\underset{\left|z\right|=1}{Max}\left|P(z)\right|,
\end{equation}
where as if $P(z)$ has all its zeros in $|z|\leq k$ where $k\leq 1,$ then
\begin{equation}\label{e5}
\underset{\left|z\right|=1}{Max}\left|P^{\prime}(z)\right|\geq        \frac{n}{1+k}\underset{\left|z\right|=1}{Max}\left|P(z)\right|.
\end{equation}
\indent Let $D_\alpha P(z)$ denotes the polar derivative of the polynomial $P(z)$ of degree $n$ with respect to the point $\alpha,$ then
$$D_\alpha P(z)=nP(z)+(\alpha-z)P^{\prime}(z). $$
The polynomial $D_\alpha P(z)$ is a polynomial of degree at most $n-1$ and it generalizes the ordinary derivative in the sense that
$$\underset{\alpha\rightarrow\infty}{Lim}\left[\dfrac{D_\alpha P(z)}{\alpha}\right]=P^{\prime}(z). $$
Now corresponding to a given $n^{th}$ degree polynomial $P(z),$ we construct a sequence of polar derivatives
\begin{equation*}
D_{\alpha_1}P(z)=nP(z)+(\alpha_1-z)P^{\prime}(z)=P_1(z)
\end{equation*}
\begin{align*}
D_{\alpha_s}D_{\alpha_{s-1}}\cdots D_{\alpha_2}D_{\alpha_1}P(z)=(n-s+1)&\left\{D_{\alpha_{s-1}}\cdots D_{\alpha_2}D_{\alpha_1}P(z)\right\}\\&+(\alpha_s-z)\left\{D_{\alpha_{s-1}}\cdots D_{\alpha_2}D_{\alpha_1}P(z)\right\}^{\prime}.
\end{align*}
The points $\alpha_1,\alpha_2,\cdots,\alpha_s,$  $s=1,2,\cdots,n,$ may be equal or unequal complex numbers.
The $s^{th}$ polar derivative $D_{\alpha_s}D_{\alpha_{s-1}}\cdots D_{\alpha_2}D_{\alpha_1}P(z)$ of $P(z)$ is a polynomial of degree at most $n-s.$ For $$P_j(z)=D_{\alpha_j}D_{\alpha_{j-1}}\cdots D_{\alpha_2}D_{\alpha_1}P(z),$$  we have
\begin{align*}
P_{j}(z)&=(n-j+1)P_{j-1}(z)+(\alpha_j-z)P^{\prime}_{j-1}(z),\,\,\,\,\,\,j=1,2,\cdots,s,\\
P_{0}(z)&=P(z).
\end{align*}
\indent A. Aziz \cite{a88} extended the inequality \eqref{e2} to the $s^{th}$ polar derivative by proving 
if $P(z)$ is a polynomial of degree $n$ not vanishing in $|z|<1,$ then for $|z|\geq 1$
\begin{align}\label{ae}
|D_{\alpha_s}\cdots &D_{\alpha_1}P(z)|\leq \dfrac{n(n-1)\cdots(n-s+1)}{2}\left\{|\alpha_1\cdots\alpha_sz^{n-s}|+1\right\}\underset{|z|=1}{Max}|P(z)|,
\end{align}
where $|\alpha_j|\geq 1,$ for $j=1,2,\cdots,s.$ The result is best possible and equality holds for the polynomial $P(z)=(z^n+1)/2.$\\

As a refinement of inequality \eqref{ae}, Aziz and Wali Mohammad \cite{aw} proved if $P(z)$ is a polynomial of degree $n$ not vanishing in $|z|<1,$ then for $|z|\geq 1$
\begin{align}\label{awe}\nonumber
|D_{\alpha_s}\cdots D_{\alpha_1}P(z)|\leq \dfrac{n(n-1)\cdots(n-s+1)}{2}&\Big\{\left(|\alpha_1\cdots\alpha_sz^{n-s}|+1\right)\underset{|z|=1}{Max}|P(z)|,\\ \,\,\,\,\,\,&-\left( |\alpha_1\cdots\alpha_sz^{n-s}|-1  \right)\Big\}\underset{|z|=1}{Min}|P(z)|
\end{align}
where $|\alpha_j|\geq 1,$ for $j=1,2,\cdots,s.$ The result is best possible and equality holds for the polynomial $P(z)=(z^n+1)/2.$ \\

\indent In this paper, we shall obtain several inequalities concerning the polar derivative of a polynomial and thereby obtain compact generalizations of inequalities \eqref{ae} and \eqref{awe}.\\

\indent We first prove following result from which certain interesting results follow as special cases.

\begin{theorem}\label{t1}
If $ F(z)$ be a polynomial of degree $n$ having all its zeros in the disk $\left|z\right|\leq k$ where $k\leq 1$, and $P(z)$ is a polynomial of degree n such that
\begin{equation*}
\left|P(z)\right|\leq \left|F(z)\right| \,\,\, for\,\,\, |z| = k,
\end{equation*}
 then for $\alpha_j,\beta\in\mathbb{C}$  with $ \left|\alpha_j\right|\geq k,\left|\beta\right|\leq 1 $, $j=1,2,\cdots,s$ and  $|z|\geq 1$,  
\begin{align}\label{te1}
\left|z^sP_s(z)+\beta\dfrac{n_s \Lambda_s}{(1+k)^s}P(z)\right|\leq \left|z^sF_s(z)+\beta\dfrac{n_s \Lambda_s}{(1+k)^s}F(z)\right|,
\end{align}
where
\begin{align}\label{te1'}
 n_s=n(n-1)\cdots(n-s+1) \,\,\&\,\,\Lambda_s=(|\alpha_1|-k)(|\alpha_2|-k)\cdots(|\alpha_s|-k).
\end{align}
\end{theorem}
If we choose $ F(z)= z^{n}M/k^{n} $, where $ M = Max_{|z|=k}\left|P(z)\right|$ in Theorem \ref{t1}, we get the following result.
\begin{corollary}\label{c1}
If $P(z)$ be a polynomial of degree $n$,   
 then for $\alpha_j,\beta\in\mathbb{C}$  with $ \left|\alpha_j\right|\geq k,\left|\beta\right|\leq 1 $,  $j=1,2,\cdots,s$ and  $|z|\geq 1$,   
\begin{align}\label{ce1}
\left|z^sP_s(z)+\beta\dfrac{n_s \Lambda_s}{(1+k)^s}P(z)\right|\leq  \dfrac{n_s|z|^n}{k^n}\left|\alpha_1\alpha_2\cdots\alpha_s+\beta\dfrac{ \Lambda_s}{(1+k)^s}\right|\underset{\left|z\right|=k}{Max}\left|P(z)\right|,
\end{align}
where $n_s$ and $\Lambda_s$ are given by \eqref{te1'}.
\end{corollary}
If $\alpha_1=\alpha_2=\cdots=\alpha_s=\alpha,$ then dividing both sides of \eqref{ce1} by $|\alpha|^s$ and letting $|\alpha|\rightarrow\infty,$ we obtain the following result.
\begin{corollary}\label{c2}
If $ P(z)$ be a polynomial of degree $n$ 
 then for $ \beta\in\mathbb{C}$  with $\left|\beta\right|\leq 1 $, and  $|z|\geq 1$,  
\begin{align}\label{ce2}
\left|z^sP^{(s)}(z)+\beta\dfrac{n_s}{(1+k)^s}P(z)\right|\leq  \dfrac{n_s|z|^n}{k^n}\left|1+\dfrac{\beta }{(1+k)^s}\right|\underset{\left|z\right|=k}{Max}\left|P(z)\right|,
\end{align}
where $n_s$  is given by \eqref{te1'}.
\end{corollary}
Again, if we take $\alpha_1=\alpha_2=\cdots=\alpha_s=\alpha,$ then divide both sides of \eqref{te1} by $|\alpha|^s$ and letting $|\alpha|\rightarrow\infty,$ we obtain the following result.
\begin{corollary}\label{c3}
If $ F(z)$ be a polynomial of degree $n$ having all its zeros in the disk $\left|z\right|\leq k$ where $k\leq 1$, and $P(z)$ is a polynomial of degree $n$ such that
\begin{equation*}
\left|P(z)\right|\leq \left|F(z)\right| \,\,\, for\,\,\, |z| = k,
\end{equation*}
 then for $\beta\in\mathbb{C}$  with $ \left|\beta\right|\leq 1 $, and  $|z|\geq 1$,  
\begin{align}\label{ce3}
\left|z^sP^{(s)}(z)+\beta\dfrac{n_s }{(1+k)^s}P(z)\right|\leq \left|z^sF^{(s)}(z)+\beta\dfrac{n_s }{(1+k)^s}F(z)\right|,
\end{align}
where $n_s$ is given by \eqref{te1'}.
\end{corollary}
If we choose $P(z)=mz^n/k,$  where $m=Min_{|z|=k}|F(z)|,$ in Theorem \ref{t1} we get the following result.
\begin{corollary}\label{c4}
If $ F(z)$ be a polynomial of degree $n$ having all its zeros in the disk $\left|z\right|\leq k$ where $k\leq 1$, 
  then for $\alpha_j,\beta\in\mathbb{C}$  with $ \left|\alpha_j\right|\geq k,\left|\beta\right|\leq 1 $, where $j=1,2,\cdots,s$ and  $|z|\geq 1$,    
\begin{align}\label{ce4}
\left|z^sF_s(z)+\beta\dfrac{n_s \Lambda_s}{(1+k)^s}F(z)\right|\geq \dfrac{n_s|z|^n}{k^n}\left|\alpha_1\alpha_2\cdots\alpha_s+\dfrac{\beta \Lambda_s}{(1+k)^s}\right|\underset{|z|=k}{Min}|F(z)|,
\end{align}
where $n_s$ and $\Lambda_s$ are given by \eqref{te1'}.
\end{corollary}
\begin{remark}
\textnormal{For $\beta=0$ and $k=1,$ we get the result due to Aziz and Wali Mohammad \cite[Theorem 1]{aw}.}
\end{remark}
Again, if we take $\alpha_1=\alpha_2=\cdots=\alpha_s=\alpha,$ then divide both sides of \eqref{ce4} by $|\alpha|^s$ and letting $|\alpha|\rightarrow\infty,$ we obtain the following result.
\begin{corollary}\label{c5}
If $ F(z)$ be a polynomial of degree $n$ having all its zeros in the disk $\left|z\right|\leq k$ where $k\leq 1$, 
 then for $ \beta\in\mathbb{C}$  with $ \left|\beta\right|\leq 1 $, and  $|z|\geq 1$,  
\begin{align}\label{ce5}
\left|z^sF^{(s)}(z)+\dfrac{\beta n_s }{(1+k)^s}F(z)\right|\geq \dfrac{n_s|z|^n}{k^n}\left|1+\dfrac{\beta }{(1+k)^s}\right|\underset{|z|=k}{Min}|F(z)|,
\end{align}
where $n_s$ is given by \eqref{te1'}.
\end{corollary}
For $s = 1,$ and $\alpha_1=\alpha$ in Theorem \ref{t1} we get the following result:
 \begin{corollary}\label{c}
 If $ F(z)$ be a polynomial of degree $n$ having all its zeros in the disk $\left|z\right|\leq k$ where $k\leq 1$, and $P(z)$ is a polynomial  n such that
 \begin{equation*}
 \left|P(z)\right|\leq \left|F(z)\right| \,\,\, for\,\,\, |z| = k,
 \end{equation*}
  then for $\alpha, \beta\in\mathbb{C}$  with $ \left|\alpha\right|\geq k,\left|\beta\right|\leq 1 $, and  $|z|\geq 1$,  
 \begin{equation}\nonumber\label{ce}
 \left|zD_\alpha P(z)+n\beta \left(\dfrac{|\alpha|-k}{k+1}\right)P(z)\right|
 \leq \left|zD_\alpha F(z)+n\beta \left(\dfrac{|\alpha|-k}{k+1}\right)F(z)\right|.
 \end{equation}
 \end{corollary}
\begin{theorem}\label{t2}
If $P(z)$ is a polynomial of degree $n,$ which does not vanish in the disk $|z|< k$ where $k\leq 1,$ then for $\alpha_1,\alpha_2,\cdots,\alpha_s,\beta\in\mathbb{C}$ with $|\alpha_1|\geq k,|\alpha_2|\geq k,\cdots,|\alpha_s|\geq k,|\beta|\leq 1$ and $|z|\geq 1$ then
\begin{align}\nonumber\label{te2}
\Bigg|z^sP_s(z)&+\beta\dfrac{n_s \Lambda_s}{(1+k)^s}P(z)\Bigg|\\&\leq \dfrac{n_s}{2}\left\{\dfrac{|z^n| }{k^n}\Bigg|\alpha_1\alpha_2\cdots\alpha_s+\dfrac{\beta \Lambda_s}{(1+k)^s} \Bigg|+\Bigg|z^s+\dfrac{\beta \Lambda_s}{(1+k)^s}\Bigg|\right\}\underset{|z|=k}{Max}|P(z)|,
\end{align}
where $n_s$ and $\Lambda_s$ are given by \eqref{te1'}.
\end{theorem}
\begin{remark}
\textnormal{If we take $\beta=0$ and $k=1,$ we get inequality \eqref{ae}.}
\end{remark}
We next prove the following refinement of Theorem \ref{t2}.
\begin{theorem}\label{t3}
If $P(z)$ is a polynomial of degree $n$ and does not vanish in the disk $|z|< k$ where $k\leq 1,$ then for $\alpha_1,\alpha_2,\cdots,\alpha_s,\beta\in\mathbb{C}$ with $|\alpha_1|\geq k,|\alpha_2|\geq k,\cdots,|\alpha_s|\geq k,|\beta|\leq 1$ and $|z|\geq 1,$ we have 
\begin{align}\nonumber\label{te3}
\Bigg|z^s&P_s(z)+\beta\dfrac{ n_s \Lambda_s}{(1+k)^s} P(z)\Bigg|\\\nonumber\leq & \dfrac{n_s}{2}\Bigg[\left\{\dfrac{|z|^n }{k^n}\Bigg|\alpha_1\alpha_2\cdots\alpha_s+\beta\dfrac{ \Lambda_s}{(1+k)^s} \Bigg|+\Bigg|z^s+\beta\dfrac{ \Lambda_s}{(1+k)^s}\Bigg|\right\}\underset{|z|=k}{Max}|P(z)|\\ \,\,\,\,&-\left\{\dfrac{|z|^n}{k^{n}}\Bigg|\alpha_1\alpha_2\cdots\alpha_s+\beta\dfrac{ \Lambda_s}{(1+k)^s}\Bigg|-\Bigg|z^s+\beta\dfrac{ \Lambda_s}{(1+k)^s}\Bigg|\right\}\underset{|z|=k}{Min}|P(z)|\Bigg],
\end{align}
where $n_s$ and $\Lambda_s$ are given by \eqref{te1'}.
\end{theorem}
If we take $\alpha_1=\alpha_2=\cdots=\alpha_s=\alpha,$ then divide both sides of \eqref{te1} by $|\alpha|^s$ and letting $|\alpha|\rightarrow\infty,$ we obtain the following result.
 \begin{corollary}\label{c7}
If $P(z)$ is a polynomial of degree $n,$ which does not vanish in the disk $|z|< k$ where $k\leq 1,$ then for $\beta\in\mathbb{C}$ with $|\beta|\leq 1$ and $|z|\geq 1$ then
\begin{align}\nonumber\label{ce7}
\Bigg|z^sP^{(s)}(z)+\beta\dfrac{ n_s }{(1+k)^s} P(z)\Bigg|\leq&  \dfrac{n_s}{2}\Bigg[\left\{\dfrac{|z|^n}{k^n}\Bigg|1+\dfrac{\beta }{(1+k)^s} \Bigg|+\Bigg|\dfrac{\beta}{(1+k)^s}\Bigg|\right\}\underset{|z|=k}{Max}|P(z)|\\ \,\,\,\,\,&-\left\{\dfrac{|z|^n}{k^{n}}\Bigg|1+\dfrac{ \beta}{(1+k)^s}\Bigg|-\Bigg|\dfrac{ \beta}{(1+k)^s}\Bigg|\right\}\underset{|z|=k}{Min}|P(z)|\Bigg],
\end{align}
where $n_s$ is given by \eqref{te1'}.
 \end{corollary}
 For $s=1$ in Corollary \ref{c7}, we get the following result.
 \begin{corollary}\label{c8}
 If $P(z)$ is a polynomial of degree $n,$ which does not vanish in the disk $|z|< k$ where $k\leq 1,$ then for $\beta\in\mathbb{C}$ with $|\alpha|\geq k,|\beta|\leq 1$ and $|z|\geq 1$ then
 \begin{align}\nonumber\label{ce8}
 \Bigg|zP^{\prime}(z)+\dfrac{ n\beta }{1+k} P(z)\Bigg|\leq & \dfrac{n}{2}\Bigg[\left\{\dfrac{|z|^n}{k^n}\Bigg|1+\dfrac{\beta }{1+k} \Bigg|+\Bigg|\dfrac{\beta}{1+k}\Bigg|\right\}\underset{|z|=k}{Max}|P(z)|\\&-\left\{\dfrac{|z|^n}{k^{n}}\Bigg|1+\dfrac{ \beta}{1+k}\Bigg|-\Bigg|\dfrac{ \beta}{1+k}\Bigg|\right\}\underset{|z|=k}{Min}|P(z)|\Bigg],
 \end{align}
  \end{corollary}
  For $\beta=0,$ Theorem \ref{t3} reduces to the following result.
 \begin{corollary}\label{c9}
 If $P(z)$ is a polynomial of degree $n,$ which does not vanish in the disk $|z|< k$ where $k\leq 1,$ then for $\alpha_1,\alpha_2,\cdots,\alpha_s,\beta\in\mathbb{C}$ with $|\alpha_1|\geq k,|\alpha_2|\geq k,\cdots,|\alpha_s|\geq k,|\beta|\leq 1$ and $|z|\geq 1$ then
 \begin{align}\label{ce9}\nonumber
 \big|P_s(z)\big|\leq \dfrac{n_s}{2}\Bigg[&\left\{\dfrac{|z|^{n-s} }{k^n}|\alpha_1\alpha_2\cdots\alpha_s|+1\right\}\underset{|z|=k}{Max}|P(z)|\\&-\left\{\dfrac{|z|^{n-s}}{k^{n}}|\alpha_1\alpha_2\cdots\alpha_s|-1\right\}\underset{|z|=k}{Min}|P(z)|\Bigg],
 \end{align}
 where $n_s$ is given by \eqref{te1'}.
 \end{corollary} 
 \begin{remark}
 \textnormal{For $k=1,$ inequality \eqref{ce9} reduces to \eqref{awe}.}
 \end{remark}
 If we take $\alpha_1=\alpha_2=\cdots=\alpha_s=\alpha,$ then divide both sides of \eqref{te1} by $|\alpha|^s$ and letting $|\alpha|\rightarrow\infty,$ we obtain the following result.
 \begin{corollary}\label{c10}
  If $P(z)$ is a polynomial of degree $n,$ which does not vanish in the disk $|z|< k$ where $k\leq 1,$ then for $\alpha_1,\alpha_2,\cdots,\alpha_s,\beta\in\mathbb{C}$ with $|\alpha_1|\geq k,|\alpha_2|\geq k,\cdots,|\alpha_s|\geq k,|\beta|\leq 1$ and $|z|\geq 1$ then
  \begin{align}\label{ce10}
  \big|P^{(s)}(z)\big|\leq \dfrac{n(n-1)\cdots(n-t+1)|z|^{n-s}}{2k^{n}}\Bigg[\underset{|z|=k}{Max}|P(z)|-\underset{|z|=k}{Min}|P(z)|\Bigg].
  \end{align}
  \end{corollary} 
  If $s=1$ and $\alpha_1=\alpha$ then inequality \eqref{te3} reduces to the following result.
  \begin{corollary}\label{c11}
If $P(z)$ is a polynomial of degree $n,$ which does not vanish in the disk $|z|< k$ where $k\leq 1,$ then for $\alpha,\beta\in\mathbb{C}$ with $|\alpha|\geq k,|\beta|\leq 1$ and $|z|\geq 1$ then
\begin{align}\nonumber\label{ce11}
\Bigg|z&D_\alpha P(z)+n\beta \left(\dfrac{ |\alpha|-k}{1+k}\right) P(z)\Bigg|\\\nonumber&\leq  \dfrac{n}{2}\Bigg[\left\{\dfrac{1 }{k^n}\Bigg|\alpha+\beta\left(\dfrac{ |\alpha|-k}{1+k}\right) \Bigg|+\Bigg|z^s+\beta\left(\dfrac{ |\alpha|-k}{1+k}\right)\Bigg|\right\}\underset{|z|=k}{Max}|P(z)|\\&-\left\{\dfrac{1}{k^{n}}\Bigg|\alpha+\beta\left(\dfrac{ |\alpha|-k}{1+k}\right)\Bigg|-\Bigg|z^s+\beta\left(\dfrac{ |\alpha|-k}{1+k}\right)\Bigg|\right\}\underset{|z|=k}{Min}|P(z)|\Bigg].
\end{align}
  \end{corollary}
\section{\bf Lemmas}
For the proof of Theorems, we need the following Lemmas.
The first Lemma follows by repeated application of Laguerre's theorem \cite{al} or\cite[p. 52]{marden}.
\begin{lemma}\label{l1}
If all the zeros of $nth$ degree polynomial lie in circular region  $C$ and if none of the points $\alpha_1,\alpha_2,\cdots,\alpha_s$ lie in circular region $C$, then each of the polar derivatives
\begin{equation}\label{le1}
D_{\alpha_s}D_{\alpha_{s-1}}\cdots D_{\alpha_1} P(z)=P_s(z),\,\,\,\,\,\,\,\,\,s=1,2,\cdots,n-1
\end{equation}
has all its zeros in $C.$
\end{lemma}
The next Lemma is due to Aziz and Rather \cite{ar98}.
\begin{lemma}\label{l2}
If $P(z)$ is a polynomial of degree $n,$ having all zeros in the closed disk $|z|\leq k,$ $ k\leq 1,$ then for every real or complex number $\alpha$ with $|\alpha|\geq k$ and $|z|=1$, we have
\begin{equation}\label{le2}
|D_\alpha P(z)|\geq n\left(\dfrac{|\alpha|-k}{1+k}\right)|P(z)|.
\end{equation}
\end{lemma}
\begin{lemma}\label{l3}
If $P(z)=\sum_{j=0}^{n}a_jz^j$ is a polynomial of degree $n$ having all its zeros in $|z|\leq k,$ $k\leq 1$ then
\begin{equation}\label{le3}
\dfrac{1}{n}\left|\dfrac{a_{n-1}}{a_n}\right|\leq k.
\end{equation}
\end{lemma}
The above lemma follows by taking $\mu=1$ in the result due to Aziz and Rather \cite{ar04}.
\begin{lemma}\label{l4}
If $P(z)$ be a polynomial of degree $n$ having all zeros in the disk $|z|\leq k$ where $k\leq 1,$ then for $\alpha_1,\alpha_2,\cdots,\alpha_s\in\mathbb{C}$ with $|\alpha_1|\geq k,|\alpha_2|\geq k,\cdots,|\alpha_s|\geq k,$ $(1\leq s<n),$ and $|z|=1$
\begin{align}\label{le4}
|P_s(z)|\geq\dfrac{n_s \Lambda_s}{(1+k)^s}|P(z)|,
\end{align}
where $n_s$ and $\Lambda_s$ are defined in \eqref{te1'}.
\end{lemma}
\begin{proof}
The result is trivial if $|\alpha_j|=k$ for at least one $j$ where $j=1,2,\cdots,t.$ Therefore, we assume $|\alpha_j|>k$ for all $j=1,2,\cdots,t.$ We shall prove Lemma by principle of mathematical induction. For $t=1$ the result is follows by Lemma \ref{l2}. \\
We assume that the result is true for $t=q,$ which means that for $|z|=1,$ we have
\begin{equation}\label{l4pe1}
|P_q(z)|\geq \dfrac{n_qA_{\alpha_q}}{(1+k)^q}|P(z)|\,\,\,q\geq 1,
\end{equation}
and we will prove that the Lemma is true for $t=q+1$ also.\\
Since $D_{\alpha_1}P(z)=(na_n\alpha_1+a_{n-1})z^{n-1}+\cdots+(na_0+\alpha_1a_1)$ and $|\alpha_1|>k,$  $D_{\alpha_1}P(z)$ is a polynomial of degree $n-1.$ If this is not true, then  $$ na_n\alpha_1+a_{n-1}=0, $$ which implies 
$$ |\alpha_1|=\dfrac{1}{n}\left|\dfrac{a_{n-1}}{a_n}\right|.   $$
By Lemma \ref{l3}, we have
$$ |\alpha_1|=\dfrac{1}{n}\left|\dfrac{a_{n-1}}{a_n}\right|\leq k .  $$
But this is the contradiction to the fact $|\alpha|>k.$ Hence,  $D_{\alpha_1}P(z)$ is a polynomial of degree $n-1$ and by Lemma \ref{l1}, $D_{\alpha_1}P(z)$ has all its zeros in $|z|\leq k.$
By the similar argument as before, $D_{\alpha_2}D_{\alpha_1}P(z)$ must be a polynomial of degree $n-2$ for $|\alpha_1|>k,$ $|\alpha_2|>k$ and all its zeros in $|z|\leq k.$ Continuing in this way, we conclude $D_{\alpha_q}D_{\alpha_{q-1}}\cdots D_{\alpha_1}P(z)=P_q(z)$ is a polynomial of degree $n-q$ for all $|\alpha_j|>k,$ $j=1,2,\dots,q$ and has all zeros in $|z|\leq k.$  Applying Lemma \ref{l2} to $P_q(z),$ we get for $|\alpha_{q+1}|>k,$ 
\begin{equation}\label{l4pe2}
|P_{q+1}(z)|=|D_{\alpha_{q+1}}P_q(z)|\geq \dfrac{(n-q)(|\alpha_{q+1}|-k)}{1+k}|P_q(z)|\,\,\,\textnormal{for}\,\,\,\,|z|=1.
\end{equation}
Inequality \eqref{l4pe2} in conjunction with \eqref{l4pe1} gives for $|z|=1,$
\begin{equation}
|P_{q+1}(z)|\geq \dfrac{n_{q+1}A_{\alpha_{q+1}}}{(1+k)^{q+1}}|P(z)|,
\end{equation}
where $n_{q+1}=n(n-1)\cdots(n-q)$ and $A_{\alpha_{q+1}}=(|\alpha_1|-k)(|\alpha_2|-k)\cdots(|\alpha_{q+1}|-k).$\\
This shows that the result is true for $s=q+1$ also. This completes the proof of Lemma \ref{l4}.
\end{proof}
\begin{lemma}\label{l5}
If $P(z)$ be a polynomial of degree $n,$  then for $\alpha_1,\alpha_2,\cdots,\alpha_s\in\mathbb{C}$ with $|\alpha_1|\geq k,|\alpha_2|\geq k,\cdots,|\alpha_s|\geq k,$ $(1\leq t<n),$ $|\beta|\leq 1$ and $|z|\geq 1,$ 
\begin{align}\nonumber\label{le5}
\Bigg|z^s&P_s(z)+\beta\dfrac{n_s \Lambda_s}{(1+k)^s}P(z)\Bigg|+k^n\left|z^sQ_s(z/k^2)+\beta\dfrac{n_s \Lambda_s}{(1+k)^s}Q(z/k^2)\right|\\\leq &n_s\left\{\dfrac{|z^n| }{k^n}\Bigg|\alpha_1\alpha_2\cdots\alpha_s+\dfrac{\beta \Lambda_s}{(1+k)^s} \Bigg|+\Bigg|z^s+\dfrac{\beta \Lambda_s}{(1+k)^s}\Bigg|\right\}\underset{|z|=k}{Max}|P(z)|,
\end{align}
where $n_s$ and $\Lambda_s$ are defined in \eqref{te1'} and $Q(z)=z^n\overline{P(1/\overline{z})}.$
\end{lemma}
\begin{proof}
Let $M=\underset{|z|=k}{Max}|P(z)|.$ Therefore, for every $\lambda$ with $|\lambda|>1,$   $|P(z)|<|\lambda Mz^n/k^n|$ on $|z|=k.$ By Rouche's theorem it follows that all the zeros of $F(z)=P(z)+\lambda Mz^n/k^n$ lie in $|z|<k.$ If $G(z)=z^n\overline{F(1/\overline{z})}$ then $|k^nG(z/k^2)|=|F(z)|$ for $|z|=k$ and hence for any $\delta$ with $|\delta|>1,$ the polynomial $H(z)= k^nG(z/k^2)+\delta F(z)$ has all its zeros in $|z|<k.$ By applying Lemma \ref{l4} to $H(z),$ we have for $\alpha_1,\alpha_2,\cdots,\alpha_s\in\mathbb{C}$ with $|\alpha_1|> k,|\alpha_2|> k,\cdots,|\alpha_s|> k,$ $(1\leq t<n),$ 
\begin{equation*}
|z^sH_s(z)|\geq\dfrac{n_s \Lambda_s}{(1+k)^s}|H(z)| \,\,\,\,\,\textnormal{for}\,\,\,\,|z|=1.
\end{equation*}
Therefore, for any $\beta$ with $|\beta|<1$ and $|z|=1,$ we have
\begin{equation*}
|z^sH_s(z)|>|\beta|\dfrac{n_s \Lambda_s}{(1+k)^s}|H(z)| .
\end{equation*}
Since $|\alpha_j|> k,$ $j=1,2,\cdots,s$ by Lemma \ref{l1} the polynomial $ z^sH_s(z)$ has all its zeros in $|z|<1$ and by Rouche's theorem, the polynomial 
$$ T(z)=z^sH_s(z)+\beta\dfrac{n_s \Lambda_s}{(1+k)^s} H(z)  $$
 has all its zeros in $|z|<1.$ Replacing $H(z)$ by $k^nG(z/k^2)+\delta F(z),$ we conclude that the polynomial
$$ T(z)=k^n\left\{ z^sG_s(z/k^2)+\dfrac{\beta n_s \Lambda_s}{(1+k)^s} G(z/k^2)\right\}+\delta\left\{ z^sF_s(z)+\dfrac{\beta n_s \Lambda_s}{(1+k)^s} F(z)\right\}$$
has all its zero in $|z|<1.$ This gives for $|\beta|<1,$ $|\alpha_j|\geq k,$ where $j=1,2,\cdots,t$ and $|z|\geq 1.$
\begin{equation}\label{le34}
k^n\left|z^sG_s(z/k^2)+\dfrac{\beta n_s \Lambda_s}{(1+k)^s} G(z/k^2)\right|\leq \left|z^sF_s(z)+\dfrac{\beta n_s \Lambda_s}{(1+k)^s} F(z)\right|
\end{equation}
If inequality \eqref{le34} is not true, then there is a point $z_0$ with $|z_0|\geq 1$ such that
\begin{equation*}
k^n\left|z_0^sG_s(z_0/k^2)+\dfrac{\beta n_s \Lambda_s}{(1+k)^s} G(z_0/k^2)\right|> \left|z_0^sF_s(z_0)+\dfrac{\beta n_s \Lambda_s}{(1+k)^s} F(z_0)\right|
\end{equation*}
Since all the zeros of $F(z)$ lie in $|z|<k,$ proceeding similarly as in case of $H(z),$ it follows that the polynomial $z^sF_s(z)+\dfrac{\beta n_s \Lambda_s}{(1+k)^s} F(z)$ has all its zeros in $|z|< 1,$ and hence $z_0^sF_s(z_0)+\dfrac{\beta n_s \Lambda_s}{(1+k)^s} F(z_0)\neq 0.$ Now, choosing
$$\delta=-\dfrac{k^n\left\{z_0^sG_s(z_0/k^2)+\dfrac{\beta n_s \Lambda_s}{(1+k)^s} G(z_0/k^2)\right\}}{z_0^sF_s(z_0)+\dfrac{\beta n_s \Lambda_s}{(1+k)^s} F(z_0)}.    $$
then $\delta$ is a well-defined real or complex number, with $|\delta|>1$ and $T(z_0)=0,$ which contradicts the fact that $T(z)$ has all zeros in $|z|<1.$ Thus, \eqref{le34} holds. Now, replacing $F(z)$ by $P(z)+\lambda Mz^n/k^n$ and $G(z)$ by $Q(z)+\overline{\lambda}M/k^n$ in \eqref{le34}, we have for $|z|\geq 1,$
\begin{align}\nonumber\label{le35}
k^n\Bigg|z^s&Q_s(z/k^2)+\dfrac{\beta n_s \Lambda_s}{(1+k)^s} Q(z/k^2)+\dfrac{\overline{\lambda}n_s}{k^n}\left\{z^s+\dfrac{\beta \Lambda_s}{(1+k)^s}\right\}M\Bigg|\\&\leq\Bigg|z^sP_s(z)+\dfrac{\beta n_s \Lambda_s}{(1+k)^s} P(z)+\dfrac{\lambda n_s}{k^n}\left\{\alpha_1\alpha_2\cdots\alpha_s+\dfrac{\beta \Lambda_s}{(1+k)^s} \right\}Mz^n\Bigg|.
\end{align}
Choosing the argument of $\lambda$ in the right hand side of \eqref{le35} such that
\begin{align*}
\Bigg|z^s&P_s(z)+\dfrac{\beta n_s \Lambda_s}{(1+k)^s} P(z)+\dfrac{\lambda n_s}{k^n}\left\{\alpha_1\alpha_2\cdots\alpha_s+\dfrac{\beta \Lambda_s}{(1+k)^s} \right\}Mz^n\Bigg|\\=&\Bigg|\dfrac{\lambda n_s}{k^n}\left\{\alpha_1\alpha_2\cdots\alpha_s+\dfrac{\beta \Lambda_s}{(1+k)^s} \right\}Mz^n\Bigg|-\Bigg|z^sP_s(z)+\dfrac{\beta n_s \Lambda_s}{(1+k)^s} P(z)\Bigg|.
\end{align*}
which is possible by Corollary \ref{c1}, then for $|z|\geq 1,$ we have
\begin{align}\nonumber\label{le36}
\Bigg|z^s&P_s(z)+\dfrac{\beta n_s \Lambda_s}{(1+k)^s} P(z)\Bigg|+k^n\Bigg|z^sQ_s(z/k^2)+\dfrac{\beta n_s \Lambda_s}{(1+k)^s} Q(z/k^2)\Bigg|\\\leq&|\lambda|n_s\Bigg[\dfrac{|z^n| }{k^n}\Bigg|\alpha_1\alpha_2\cdots\alpha_s+\dfrac{\beta \Lambda_s}{(1+k)^s} \Bigg|+\Bigg|z^s+\dfrac{\beta \Lambda_s}{(1+k)^s}\Bigg|\Bigg]M
\end{align}
Making $|\lambda|\rightarrow 1$ and using the continuity for $|\beta|=1$ and $|\alpha_j|=k,$ where $j=1,2,\cdots,t,$ in \eqref{le36}, we get the inequality \eqref{le5}.
\end{proof}
\section{\textbf{Proof of Theorems}}
\begin{proof}[\textnormal{\textbf{Proof of Theorem \ref{t1}}}]
By hypothesis $F(z)$ is a polynomial of degree $n$  having all in zeros in the closed disk $|z|\leq k $ and $P(z)$ is a polynomial of degree $n$ such that 
  \begin{equation}\label{t1e1}
  |P(z)|\leq |F(z)|\,\,\,\, \textrm{for} \,\,\,\,|z|= k,
  \end{equation}
  therefore, if $F(z)$ has a zero of multiplicity $s$ at $z=ke^{i\theta_{0}}$, then $P(z)$ has a zero of multiplicity at least $s$ at $z=ke^{i\theta_{0}}$. If $P(z)/F(z)$ is a constant, then inequality \eqref{t1e1} is obvious. We now assume that $P(z)/F(z)$ is not a constant, so that by the maximum modulus principle, it follows that\\
  \[|P(z)|<|F(z)|\,\,\,\textrm{for}\,\, |z|>k\,\,.\]
  Suppose $F(z)$ has $m$ zeros on $|z|=k$ where $0\leq m < n$, so that we can write\\
  \[F(z) = F_{1}(z)F_{2}(z)\]
  where $F_{1}(z)$ is a polynomial of degree $m$ whose all zeros lie on $|z|=k$ and $F_{2}(z)$ is a polynomial of degree exactly $n-m$ having all its zeros in $|z|<k$. This implies with the help of inequality \eqref{t1e1} that\\
  \[P(z) = P_{1}(z)F_{1}(z)\]
  where $P_{1}(z)$ is a polynomial of degree at most $n-m$. Again, from inequality \eqref{t1e1}, we have
  \[|P_{1}(z)| \leq |F_{2}(z)|\,\,\,for \,\, |z|=k\,\]
  where $F_{2}(z) \neq 0 \,\, for\,\, |z|=k$. Therefore
   for every real or complex number $\lambda $ with $|\lambda|>1$, a direct application of Rouche's theorem shows that the zeros of the polynomial $P_{1}(z)- \lambda F_{2}(z)$ of degree $n-m \geq 1$ lie in $|z|<k$ hence the polynomial 
   \[G(z) = F_{1}(z)\left(P_{1}(z) - \lambda F_{2}(z)\right)=P(z) - \lambda F(z)\]
   has all its zeros in $|z|\leq k.$ Therefore for $r> 1,$ all the zeros of $G(rz)$ lie in $|z|\leq k/r< k$
   By applying Lemma \ref{l4} to the polynomial $G(rz),$ then for $|z|=1,$ we have
   \begin{equation*}
|z^sG_s(rz)|\geq\dfrac{n_s \Lambda_s}{(1+k)^s}|G(rz)|.
   \end{equation*}
   Equivalently for $|z|=1,$ we have
   \begin{equation}
  |z^sP_s(rz)-\lambda z^s F_s(rz)|\geq\dfrac{n_s \Lambda_s}{(1+k)^s}|P(rz)-\lambda F(rz)|.
      \end{equation}
      Therefore, we have for any $\beta$ with $|\beta|<1$ and $|z|=1,$
 \begin{equation}
|z^sP_s(rz)-\lambda z^sF_s(rz)|>\dfrac{n_s \Lambda_s}{(1+k)^s}|\beta||P(rz)-\lambda F(rz)|.
\end{equation}
Since $|\alpha_j|>k,$ $j=1,2,\cdots,t$ by Rouche's theorem, the polynomial
\begin{align*}
 T(rz)&= \left\{ z^sP_s(rz)-\lambda z^sF_s(rz)\right\}+\dfrac{n_s \Lambda_s}{(1+k)^s}\beta \left\{P(rz)-\lambda F(rz)\right\}\\&     
= z^sP_s(rz)+\dfrac{n_s \Lambda_s\beta}{(1+k)^s}P(rz)- \lambda\left\{z^sF_s(rz)+\dfrac{n_s \Lambda_s\beta}{(1+k)^s}F(rz) \right\}
\end{align*}
has all zeros in $|z|< 1.$ 
This implies for $|z|\geq 1$
\begin{equation}\label{t1p1}
|z^sP_s(rz)+\dfrac{n_s \Lambda_s\beta}{(1+k)^s}P(rz)|\leq |z^sF_s(rz)+\dfrac{n_s \Lambda_s\beta}{(1+k)^s}F(rz)|
\end{equation}
If it is true, then we have for some $z_0$ with $|z_0|\geq k,$
\begin{equation}
|z_0^sP_s(rz_0)+\dfrac{n_s \Lambda_s\beta}{(1+k)^s}P(rz_0)|> |z_0^sF_s(rz_0)+\dfrac{n_s \Lambda_s\beta}{(1+k)^s}F(rz_0)|
\end{equation}
Since $z_0^sP_s(rz_0)+\dfrac{n_s \Lambda_s\beta}{(1+k)^s}P(rz_0)\neq 0,$ we can choose 
$$  \lambda=\dfrac{z_0^sF_s(rz_0)+\dfrac{n_s \Lambda_s\beta}{(1+k)^s}F(rz_0)}{z_0^sP_s(rz_0)+\dfrac{n_s \Lambda_s\beta}{(1+k)^s}P(rz_0)}    $$
$\lambda$ is well-defined real or complex number with $|\lambda|<1$ and with this choice of $\lambda,$ and $T(rz_0)=0$ for $|z_0|\geq k.$ But this is a contradiction to the fact that $T(rz)\neq 0$ for $|z|\geq 1.$ Thus \eqref{t1p1} holds. Letting $r\rightarrow 1,$ in \eqref{t1p1}, we get the desired result.
\end{proof}
\begin{proof}[\textnormal{\textbf{Proof of Theorem \textbf{\ref{t2}}}}]
 Let  $Q(z)=z^{n}\overline{P(1/\overline{z})}.$ Since $ P(z) $ does not vanish in the disk $ |z|<k,\,\, k\leq 1 $, then the polynomial $Q(z/k^2)$ has all zeros in $|z|\leq k.$  Applying Theorem \ref{t1} to $k^nQ(z/k^2)$ and noting that $|P(z)|=|k^nQ(z/k^2)|$ for $|z|=k,$ we have for all $ \alpha_j,\beta\in\mathbb{C} $ with $ |\alpha_j|\geq 1,$ $j=1,2,\cdots,t,$ $|\beta|\leq 1 ,$ and $ |z|\geq 1 ,$  
 \begin{align}\label{t2e1}
 \left|z^sP_s(z)+\beta\dfrac{n_s \Lambda_s}{(1+k)^s}P(z)\right|\leq k^n\left|z^sQ_s(z/k^2)+\beta\dfrac{n_s \Lambda_s}{(1+k)^s}Q(z/k^2)\right|.
 \end{align}
 Inequality \eqref{t2e1} in conjunction with Lemma \ref{l5} gives for all $ \alpha_j,\beta\in\mathbb{C} $ with $ |\alpha_j|\geq k,$ $j=1,2,\cdots,t$  $|\beta|\leq 1, $  and $ |z|\geq 1 ,$ 
\begin{align*}
2\Bigg|&z^sP_s(z)+\beta\dfrac{n_s \Lambda_s}{(1+k)^s}P(z)\Bigg|\\&\leq  \left|z^sP_s(z)+\beta\dfrac{n_s \Lambda_s}{(1+k)^s}P(z)\right|+ k^n\left|z^sQ_s(z/k^2)+\beta\dfrac{n_s \Lambda_s}{(1+k)^s}Q(z/k^2)\right|\\&\leq n_s\left\{\dfrac{|z^n| }{k^n}\Bigg|\alpha_1\alpha_2\cdots\alpha_s+\dfrac{\beta \Lambda_s}{(1+k)^s} \Bigg|+\Bigg|z^s+\dfrac{\beta \Lambda_s}{(1+k)^s}\Bigg|\right\}\underset{|z|=k}{Max}|P(z)|,
\end{align*}
which is equivalent to \eqref{te2}.
\end{proof}
\begin{proof}[\textnormal{\textbf{Proof of Theorem \ref{t3}}}]
Let $m=Min_{|z|=k}|P(z)|.$ If $P(z)$ has a zero on $|z|=k,$ then $m=0$ and result follows from Theorem \ref{t2}. Therefore,  We assume that $ P(z) $ has all its zeros in $ |z|>k $ where $ k\leq 1 $ so that $ m > 0 $. Now for every  $ \lambda $ with $ |\lambda|<1 $, it follows by Rouche's theorem that $ h(z)=P(z)-\lambda m $ does not vanish in $ |z|<k $. Let $g(z)=z^n\overline{h(1/\overline{z})}=z^n\overline{P(1/\overline{z})}-\overline{\lambda}mz^n=Q(z)-\overline{\lambda}mz^n$ then, the polynomial $g(z/k^2)$ has all its zeros in $|z|\leq k.$ As $|k^ng(z/k^2)|=|h(z)|$ for $|z|=k,$ applying Theorem \ref{t1} to $k^ng(z/k^2),$ we get for $\alpha_1,\alpha_2,\cdots,\alpha_s,\beta\in\mathbb{C},$ with $|\alpha_j|\geq k,$ $j=1,2,\cdots,t$ $|\beta|\leq 1$ and $|z|\geq 1,$
\begin{align}\label{t3e1}
\left|z^sh_s(z)+\beta\dfrac{n_s \Lambda_s}{(1+k)^s}h(z)\right|\leq k^n\left|z^sg_s(z/k^2)+\beta\dfrac{n_s \Lambda_s}{(1+k)^s}g(z/k^2)\right|.
\end{align}
Equivalently for $|z|\geq 1,$ we have
\begin{align}\nonumber\label{t3e2}
\Bigg|z^s&P_s(z)+\dfrac{\beta n_s \Lambda_s}{(1+k)^s} P(z)-\lambda n_s\left\{z^s+\dfrac{\beta \Lambda_s}{(1+k)^s}\right\}m\Bigg|\\&\leq k^n\Bigg|z^sQ_s(z/k^2)+\dfrac{\beta n_s \Lambda_s}{(1+k)^s} Q(z/k^2)-\dfrac{\overline{\lambda}n_s}{k^{2n}}\left\{\alpha_1\alpha_2\cdots\alpha_s+\dfrac{\beta \Lambda_s}{(1+k)^s}\right\}mz^n\Bigg|.
\end{align}
Since $Q(z/k^2)$ has all its zeros in $|z|\leq k$ and $k^n\underset{|z|=k}{Min}|Q(z/k^2)|=\underset{|z|=k}{Min}|P(z)|,$ by Corollary \ref{c4} applied to $Q(z/k^2),$ we have for $|z|\geq 1,$
\begin{align}\nonumber\label{t3e3}
\Bigg|z^sQ_s(z/k^2)+\beta\dfrac{n_s \Lambda_s}{(1+k)^s}Q(z/k^2)\Bigg|&\geq \dfrac{n_s}{k^n}\left|\alpha_1\alpha_2\cdots\alpha_s+\dfrac{\beta \Lambda_s}{(1+k)^s}\right|\underset{|z|=k}{Min}|Q(z)|\\\nonumber&= \dfrac{n_s}{k^{2n}}\left|\alpha_1\alpha_2\cdots\alpha_s+\dfrac{\beta \Lambda_s}{(1+k)^s}\right|\underset{|z|=k}{Min}|P(z)|\\&= \dfrac{n_s}{k^{2n}}\left|\alpha_1\alpha_2\cdots\alpha_s+\dfrac{\beta \Lambda_s}{(1+k)^s}\right|m.
\end{align}
Now, choosing the argument of $\lambda$ on the right hand side of inequality \eqref{t3e2} such that
\begin{align*}
 k^n&\Bigg|z^sQ_s(z/k^2)+\dfrac{\beta n_s \Lambda_s}{(1+k)^s} Q(z/k^2)-\dfrac{\overline{\lambda}n_s}{k^{2n}}\left\{\alpha_1\alpha_2\cdots\alpha_s+\dfrac{\beta \Lambda_s}{(1+k)^s}\right\}mz^n\Bigg|\\&=k^n\Bigg|z^sQ_s(z/k^2)+\dfrac{\beta n_s \Lambda_s}{(1+k)^s} Q(z/k^2)\Bigg|-\dfrac{|\overline{\lambda}|n_s}{k^{n}}\Bigg|\left\{\alpha_1\alpha_2\cdots\alpha_s+\dfrac{\beta \Lambda_s}{(1+k)^s}\right\}mz^n\Bigg|,
\end{align*}
which is possible by inequality \eqref{t3e3}, we get for $|z|\geq 1,$
\begin{align*}
\Bigg|z^s&P_s(z)+\dfrac{\beta n_s \Lambda_s}{(1+k)^s} P(z)\Bigg|-|\lambda| n_s\Bigg|z^s+\dfrac{\beta \Lambda_s}{(1+k)^s}\Bigg|m\\&=k^n\Bigg|z^sQ_s(z/k^2)+\dfrac{\beta n_s \Lambda_s}{(1+k)^s} Q(z/k^2)\Bigg|-\dfrac{|\lambda||z|^nn_s}{k^{n}}\Bigg|\alpha_1\alpha_2\cdots\alpha_s+\dfrac{\beta \Lambda_s}{(1+k)^s}\Bigg|m,
\end{align*}
Letting $\lambda\rightarrow 1,$ we have for $|z|\geq 1,$
\begin{align}\nonumber\label{t3e4}
\Bigg|z^s&P_s(z)+\dfrac{\beta n_s \Lambda_s}{(1+k)^s} P(z)\Bigg|-k^n\Bigg|z^sQ_s(z/k^2)+\dfrac{\beta n_s \Lambda_s}{(1+k)^s} Q(z/k^2)\Bigg|\\&\leq n_s\left\{\Bigg|z^s+\dfrac{\beta \Lambda_s}{(1+k)^s}\Bigg|-\dfrac{|z|^n}{k^{n}}\Bigg|\alpha_1\alpha_2\cdots\alpha_s+\dfrac{\beta \Lambda_s}{(1+k)^s}\Bigg|\right\}m
\end{align}
Adding \eqref{le5} and \eqref{t3e4}, we get for $|z|\geq 1$
\begin{align*}
2\Bigg|z^s&P_s(z)+\dfrac{\beta n_s \Lambda_s}{(1+k)^s} P(z)\Bigg|\\&\leq  n_s\Bigg[\left\{\dfrac{|z^n| }{k^n}\Bigg|\alpha_1\alpha_2\cdots\alpha_s+\dfrac{\beta \Lambda_s}{(1+k)^s} \Bigg|+\Bigg|z^s+\dfrac{\beta \Lambda_s}{(1+k)^s}\Bigg|\right\}\underset{|z|=k}{Max}|P(z)|\\&+\left\{\Bigg|z^s+\dfrac{\beta \Lambda_s}{(1+k)^s}\Bigg|-\dfrac{|z|^n}{k^{n}}\Bigg|\alpha_1\alpha_2\cdots\alpha_s+\dfrac{\beta \Lambda_s}{(1+k)^s}mz^n\Bigg|\right\}m\Bigg],
\end{align*}
which is equivalent to \eqref{te3}. This completes the proof of Theorem \ref{t3}.
\end{proof} 

\end{document}